\documentclass[plain]{amsart}
\usepackage[toc,page]{appendix}
\usepackage{amssymb,graphicx,epsfig,amsxtra, amsmath}
\usepackage[]{graphics}
\usepackage{amscd} 
\usepackage{xypic}
\usepackage[mathscr]{eucal}
\usepackage{xcolor}
\usepackage{a4}
\usepackage[normalem]{ulem}
\usepackage{paralist}
\usepackage{comment}
\input{xy}
\xyoption{all}
\newtheorem{theorem}{\textbf{Theorem}}[section]
\newtheorem{proposition}[theorem]{\textbf{Proposition}}
\newtheorem{lemma}[theorem]{\textbf{Lemma}}
\newtheorem{corollary}[theorem]{\textbf{Corollary}}

\newtheorem{claim}[theorem]{\textbf{Claim}}
\newtheorem{problem}[theorem]{{Problem}}
\theoremstyle{definition}
\newtheorem{definition}[theorem]{Definition}
\newtheorem{remark}[theorem]{Remark}

\newtheorem{definition-remark}{Definition-Remark}

\def\ag{\`a}

\def\deg{\operatorname{deg}}

\def\geq{\geqslant}
\def\leq{\leqslant}

\begin{document}

\title[On nodal deformations of singular surfaces in $\mathbb P^3$]
{On nodal deformations of singular surfaces in $\mathbb P^3$}
\author
{Ciro Ciliberto}
\address{Dipartimento di Matematica\\
Universit\`a di Roma ``Tor Vergata''\\
 Via della Ricerca Scientifica\\
00133 Roma, Italy.}
\email{cilibert@mat.uniroma2.it}

\author{Concettina Galati }
\address{Dipartimento di Matematica\\
 Universit\ag\, della Calabria\\
via P. Bucci, cubo 31B\\
87036 Arcavacata di Rende (CS), Italy. }
\email{concettina.galati@unical.it }




\date{}

\dedicatory{}

\commby{}


\begin{abstract}  In this paper we study nodal deformations of singular surfaces $S\subset \mathbb P^3$. In particular, we consider the case in which $S$ has an isolated singularity of multiplicity $m$ and the case in which $S$ has only ordinary singularities along a line. \end{abstract}


\maketitle

\tableofcontents

\section{Introduction} 

Let $V^{\mathbb P^3, |\mathcal O_{\mathbb P^3}(n)|}_\delta\subset |\mathcal O_{\mathbb P^3}(n)|$ be the \emph{Severi variety} of surfaces $S$ of degree $n$ with $\delta\geq 1$ ordinary double points (called \emph{nodes}) and no other singularity. Then $V^{\mathbb P^3, |\mathcal O_{\mathbb P^3}(n)|}_\delta$ is a locally closed subvariety of $|\mathcal O_{\mathbb P^3}(n)|\cong \mathbb P^{{{n+3}\choose 3}-1}$. Of course $V^{\mathbb P^3, |\mathcal O_{\mathbb P^3}(n)|}_1$ is irreducible of codimension 1 in $|\mathcal O_{\mathbb P^3}(n)|$ and its closure $\mathbb S_n\subset  |\mathcal O_{\mathbb P^3}(n)|$ is the set of all singular surfaces of degree $n$.  If $S\in \mathbb S_n$ is a singular surface, it makes sense to consider the maximum $\delta_S$ of the integers $\delta$ such that the surface $S$ can be deformed to a $\delta$-nodal surface, namely, the maximum $\delta_S$ of the integers $\delta$ such that the point [S] corresponding to $S$ belongs to the closure of $V^{\mathbb P^3, |\mathcal O_{\mathbb P^3}(n)|}_\delta$. 

In this article we will consider the following:

\begin{problem}\label{prob:main}
Let $S\in \mathbb S_n$ be a degree $n$ singular surface in $\mathbb P^3$. In the closure of which Severi varieties $V^{\mathbb P^3, |\mathcal O_{\mathbb P^3}(n)|}_\delta$ the point [S] corresponding to $S$
is contained? In particular, compute $\delta_S$. 
\end{problem}

Though this question  is  rather basic, it does not seem to have received much attention so far. 

One can ask a similar question for plane curves, and it seems to  have been  considered for the first time  by G. Albanese in \cite{Alb}, where he claimed that an irreducible plane curve of degree $n$ and genus $g$ is contained in the closure of the Severi variety $V^{\mathbb P^2, |\mathcal O_{\mathbb P^2}(n)|}_\delta$, with 
$$\delta={{n-1}\choose 2}-g$$
and $\delta$ is, of course, the maximum integer for which this happens. 
However, Albanese's arguments do not meet the present standard of rigor, and this theorem has been proved in more recent times by Arbarello--Cornalba \cite {AC} and Zariski \cite {Za}. 

In the case of surfaces in $\mathbb P^3$ the problem is much more complicated and we have no idea of what could be a general answer to it. In the present paper, which should be regarded as having an experimental flavor, we focus on two study cases, namely: (a) the case in which $S$ has a single ordinary or quasi--ordinary singularity of multiplicity $m$, cf. Definition \ref{def:ordinary}; (b) the case in which $S$ has ordinary singularities along a line, i.e., at the general point of the line $S$ has normal crossings, and at finitely many points of the line it has \emph{pinch--points} (where the local equation is: $x^2y-z^2=0$). 

In both cases (a) and (b), we explicitly construct nodal deformations of our surface $S$ using a suitable degeneration argument.

Our methods are based on deformation theory of regular components of the Severi variety $V^{\mathbb P^3, |\mathcal O_{\mathbb P^3}(n)|}_\delta$ whose closure may contain the point [S] corresponding to $S$,  as well as regular components of a suitable universal Severi variety over a family of threefolds, us explained further below.   

Recall that an irreducible component $V$ of a Severi variety $V^{\mathbb P^3, |\mathcal O_{\mathbb P^3}(n)|}_\delta$ is said to be \emph{regular} if for the general point $[S]\in V$ one has $h^1(\mathbb P^3, \mathcal O_{\mathbb P^3}(n)\otimes \mathcal I_N)=0$, where $N$ is the reduced scheme of length $\delta$ of the singular points of $S$. This means that the points in $N$ give $\delta$ independent conditions to the surfaces in $|\mathcal O_{\mathbb P^3}(n)|$. If $V$ is regular, then $V$ is reduced and 
$$
\dim (V)= h^0(\mathbb P^3, \mathcal O_{\mathbb P^3}(n))-1-\delta={{n+3}\choose 3}-1-\delta.
$$
Note that if  $[S]$ is in the closure of an irreducible regular component of $V^{\mathbb P^3, |\mathcal O_{\mathbb P^3}(n)|}_\delta$, then it is also contained in the closure of an irreducible regular component of $V^{\mathbb P^3, |\mathcal O_{\mathbb P^3}(n)|}_{\delta'}$ with $\delta'\leq \delta$, because the nodes in a regular component of the Severi variety can be independently smoothed.

Case (a) is considered in Section \ref {sec:ord}. In Corollary \ref {cor:gtp}, (i)--(iii), we study Problem \ref {prob:main} for $3\leq m\leq 5$.  This is the best possible result for $m=3$ and $n=m=4,5$. 
For $3\leq m\leq 5$ and if $S$ is a general reduced degree $n$ surface with an ordinary singularity of multiplicity $m$, Corollary \ref {cor:gtp} also gives the maximum number  $\delta_S$ of integers $\delta$ such that $S$ may be deformed to a $\delta$ nodal surface in a regular component of the Severi variety $V^{\mathbb P^3, |\mathcal O_{\mathbb P^3}(n)|}_\delta$. Precisely, somehow unexpectedly $\delta_S=4$ for $m=3$, while $\delta_S=16$ for $m=4$, and $\delta_S=31$  for $m=5$, satisfying the expectation $\delta_S=\delta_0(m)$ given by \eqref{codimension}.  In Corollary \ref {cor:gtp}, (iv),  we  also provide a partial answer to Problem \ref {prob:main} for any $m$. In particular, we prove that if $S\subset \mathbb P^3$ is surface  of degree $n$ with a unique \emph{general} ordinary singularity of multiplicity $m$, then  the point corresponding to $S$ is in the closure of a regular component of $V^{\mathbb P^3, |\mathcal O_{\mathbb P^3}(n)|}_{{m-1}\choose 2}$. An ordinary singularity of multiplicity $m$ is said to be \emph{general} if its tangent cone is the cone over a general plane curve $C$ of degree $m$. In any case, we show that the answer to Problem \ref {prob:main} in case (a)  definitely depends on the fact that the singularity is ordinary or quasi--ordinary. Compare Proposition \ref{prop:7} with Corollary \ref {cor:gtp}, (i). 

In order to attack case (a) we first need to consider in Section \ref {sec:curvesec} a preliminary, equally interesting problem, i.e., Problem \ref {prob:main2}: given  an irreducible component $V$ of a Severi variety $V^{\mathbb P^3, |\mathcal O_{\mathbb P^3}(n)|}_\delta$, and given  $[S]\in V$  a general point and $C$  a general plane curve section of $S$,  when is it the case that $C$ is a general plane curve of degree $n$? Again we are unable to solve this problem in general, but we prove some partial results, of independent interest (see, e.g., Proposition \ref {prop:good} and its consequences)   that are useful to prove the aforementioned results in case (a), thanks to Theorem \ref{thm:ord}.

Theorem \ref{thm:ord} and Proposition \ref{prop:good}, as well as Theorem \ref{thm:ord1}, relate Problem \ref {prob:main2} to Problem \ref {prob:main} in case (a), providing an effective method for constructing nodal deformations of a general reduced surface with an ordinary multiple point.

Case (b) is treated in Section \ref {sec:line}. Our main result here is Theorem \ref {thm:m3}, that says that if $S$ is a general surface of degree $n\geq 3$ that has only {ordinary singularities} along a line, then $S$ belongs to the closure of a regular component of $V^{\mathbb P^3, |\mathcal O_{\mathbb P^3}(n)|}_{3n-4-\epsilon}$, with  $0\leq \epsilon\leq 1$ having the same parity of $n$. In particular, we prove that we can choose $\ell$, with $3n-4=2\ell+\epsilon$, pinch points of $S$ on the line so that $S$ deforms to a $2\ell$--nodal surface in a regular component  of the corresponding Severi variety, so that each chosen pinch point deforms to two nodes. The integer $3n-4-\epsilon$ is the maximum  for which this happens in the $n$ even case, whereas in the $n$ odd case the maximum could be $3n-4$, instead of $3n-4-\epsilon=3n-5$.  

To prove this result we have to rely on the existence of suitable rational surfaces of degree $n$ with a line $r$ of points of multiplicity $n-2$ and with $3n-4-\epsilon$ nodes off $r$, imposing independent conditions to surfaces of degree $n$ in $\mathbb P^3$ that have points of multiplicity $n-2$ along $r$. This is done in Section \ref {sec:rat}. 

The method of the proof in both cases (a) and (b) is the same, and it is based on the following deformation argument. In both cases we construct a flat family $\mathcal X\longrightarrow \mathbb D$ parametrized by a disk $\mathbb D$, such that for $t\neq 0$, the fibre $\mathcal X_t$ over $t\in \mathbb D$ is isomorphic to $\mathbb P^3$, whereas the central fibre $\mathcal X_0$ is the union of two smooth irreducible components $A\cup \Theta$, intersecting transversally along a surface $E$. The starting surface $S$ for which we consider Problem \ref {prob:main} in the two cases (a) and (b) has a minimal resolution $f: S'\longrightarrow S$, with a curve $C$ mapped by $f$ to the singular locus of $S$. What we  do is to realize $S'$ as a surface $S_A$ in the threefold $A$ in such a way that $S_A\cap E=C$. Then we construct a surface $S_\Theta\subset \Theta$ such that $S_\Theta\cap E=C$ so that $S_0=S_A\cup S_\Theta$ is a Cartier divisor in $\mathcal X_0$. Finally we can do this so that $S_\Theta$ has $\delta$ nodes off $E$ imposing independent conditions to the surfaces in the linear system $|S_0|$ on $\mathcal X_0$. If this is the case then we can apply \cite [Thm. 4.6]{CG}, ensuring that $S_0$ can be deformed to a surface $S_t\subset \mathcal X_t\cong \mathbb P^3$ of degree $n$ with $\delta$ nodes that give independent conditions to surfaces of degree $n$. From this it follows  that $S$  is contained in the closure of a regular component of $V^{\mathbb P^3, |\mathcal O_{\mathbb P^3}(n)|}_\delta$.

We notice that the strength of this method consists in studying  deformations of a surface $S\subset \mathbb P^3$ corresponding to a (presumably) singular point of a Severi variety in terms of deformations of a surface $S_A\cup S_\Theta$ corresponding to a smooth point in  a relative Severi variety.

\medskip

{\bf Aknowledgements:} The authors are members of GNSAGA of INdAM. In particular, the second author acknowledges
funding from the GNSAGA of INdAM and the European Union - NextGenerationEU under the National Recovery and Resilience
Plan (PNRR) - Mission 4 Education and research - Component 2 From research to business - Investment 1.1, Prin 2022
``Geometry of algebraic structures: moduli, invariants, deformations", DD N. 104, 2/2/2022, proposal code 2022BTA242 - CUP
J53D23003720006. 

The authors also thank Fabrizio Catanese, Barbara Fantechi and Rita Pardini for useful discussions on related problems to the subject of this paper. \medskip

\section{Nodal surfaces with general plane curve section}\label{sec:curvesec}

Before considering Problem \ref {prob:main}, in this section we will consider the following  related problem:

\begin{problem}\label{prob:main2}
Let $V$ be an irreducible component of a Severi variety $V^{\mathbb P^3, |\mathcal O_{\mathbb P^3}(n)|}_\delta$. Let $[S]\in V$  be  general  and let $C$ be a general plane curve section of the surface $S$.  When does it occur that $C$ is a general plane curve of degree $n$?
\end{problem}

Let $V$ as above be such that the answer to Problem \ref {prob:main2} is affermative. Then we will say that $V$ is a \emph{good component} of the Severi variety. 

\begin{proposition}\label{prop:stop} Let $V$ be a good component of a Severi variety $V^{\mathbb P^3, |\mathcal O_{\mathbb P^3}(n)|}_\delta$, with $n\geq 2$. Then
\begin{equation}\label{eq:lip}
\dim(V)\geq \frac {n(n+3)}2.
\end{equation}
\end{proposition}
\begin{proof} Let $\mathcal H_n$ be the Hilbert scheme of plane curves of degree $n$ in $\mathbb P^3$. One has that $\mathcal H_n$ is irreducible and $\dim(\mathcal H_n)= \frac {n(n+3)}2+3$.

Consider the incidence variety 
\begin{equation}\label{eq:prot}
\mathbb I=\{([S],[C])\in V\times \mathcal H_n: C\subset S\}.
\end{equation}
One has that  $\mathbb I$ is irreducible of dimension  $\dim (\mathbb I)=\dim(V)+3$. On the other hand, since $V$ is  a good component,  the projection $\pi:\mathbb I\longrightarrow \mathcal H_n$ is dominant. Hence  we have 
$$\dim(V)+3=\dim(\mathbb I)\geq \mathcal H_n=\frac {n(n+3)}2+3,$$ whence the assertion follows.  \end{proof}

Recall that  the tangent space to the Severi variety $V^{\mathbb P^3, |\mathcal O_{\mathbb P^3}(n)|}_\delta$ at a point $[S]$ can be identified with
$H^0(S, \mathcal O_{ S}(n)\otimes \mathcal I_N)$, where $N$ is the reduced scheme of length $\delta$ supported at the nodes of $S$.   Recall that an irreducible component $V$ of  $V^{\mathbb P^3, |\mathcal O_{\mathbb P^3}(n)|}_\delta$ is said to be \emph{regular} if for the general point  $[S]\in V$  one has $ h^1(S, \mathcal O_{ S}(n)\otimes \mathcal I_N)= h^1(\mathbb P^3, \mathcal O_{\mathbb P^3}(n)\otimes \mathcal I_N)=0$.   This means that the points in $N$ give $\delta$ independent conditions to the surfaces in $|\mathcal O_{\mathbb P^3}(n)|$. If $V$ is regular, then $V$ is reduced and 
$$
\dim (V)= h^0(\mathbb P^3, \mathcal O_{\mathbb P^3}(n))-1-\delta={{n+3}\choose 3}-1-\delta.
$$

We will say that $V$ is \emph{very regular} if for the general  $[S]\in V$ one has $h^1(\mathbb P^3, \mathcal O_{\mathbb P^3}(n-1)\otimes \mathcal I_N)=0$, with $N$ as above the set of  nodes  of $S$. Of course, if $V$ is very regular, it is also regular. The following proposition is due to Severi in  \cite [\S 1]{Sev}:

\begin{proposition}  \label{prop:sev} Let $V$ be a very regular irreducible component $V$ of  $V^{\mathbb P^3, |\mathcal O_{\mathbb P^3}(n)|}_\delta$,
with $n\geq 3$.  Then 
\begin{equation}\label{eq:lim}
\delta \leq {{n+2}\choose 3}-4=:\delta_0(n).
\end{equation}
\end{proposition}

\begin{proof} Let  $[S]\in V$  be general  and let $N$ be its set of double points. Under the hypothesis, one has
$$
h^0(\mathbb P^3, \mathcal O_{\mathbb P^3}(n-1)\otimes \mathcal I_N)=
h^0(\mathbb P^3, \mathcal O_{\mathbb P^3}(n-1))-\delta={{n+2}\choose 3}-\delta.
$$
On the other hand, $h^0(\mathbb P^3, \mathcal O_{\mathbb P^3}(n-1)\otimes \mathcal I_N)\geq 4$, because $H^0(\mathbb P^3, \mathcal O_{\mathbb P^3}(n-1)\otimes \mathcal I_N)$ contains the first polars of $S$, which are linearly independent because $S$ is not a cone. The assertion follows. \end{proof}

\begin{proposition}\label{prop:good} Let $V$ be a  regular irreducible component of  $V^{\mathbb P^3, |\mathcal O_{\mathbb P^3}(n)|}_\delta$. Then $V$ is  good if and only if $V$ is very regular.
\end{proposition}

\begin{proof}  This result is trivially verified for $n=1$. So we may assume $n\geq 2$.

Consider again the incidence variety $\mathbb I$ as in \eqref {eq:prot} and the projection $\pi: \mathbb I\longrightarrow \mathcal H_n$. Let $([S],[C])\in \mathbb I$ be a general point, and let $N$ be, as usual, the set of $\delta$ nodes of $S$. The tangent space to the fibre of $\pi$ passing through $([S],[C])$ at $([S],[C])$ is 
$$
H^0(S,   \mathcal O_{S}(n)\otimes \mathcal I_N\otimes \mathcal I_C)=H^0(S,   \mathcal O_{S}(n-1)\otimes \mathcal I_N)=H^0(\mathbb P^3,   \mathcal O_{\mathbb P^3}(n-1)\otimes \mathcal I_N).
$$
Hence, the image of $\pi: \mathbb I\longrightarrow \mathcal H_n$ has dimension
\begin{eqnarray*}
&\dim({\rm Im}(\pi))=\dim(\mathbb I)-h^0(S, \mathcal O_{S}(n-1)\otimes \mathcal I_N)
=\\
 &=\dim(V)+3-h^0(S, \mathcal O_{S}(n-1)\otimes \mathcal I_N)=\\
&= {{n+3}\choose 3}-\delta+2-h^0(S, \mathcal O_{S}(n-1)\otimes \mathcal I_N).
\end{eqnarray*}
Now $V$ is good if and only if 
$$
\frac {n(n+3)}2 +3=\dim(\mathcal H_n)=\dim({\rm Im}(\pi))={{n+3}\choose 3}-\delta+2-h^0(S, \mathcal O_{S}(n-1)\otimes \mathcal I_N)
$$
i.e., if and only if
\begin{eqnarray*}
&h^0(\mathbb P^3, \mathcal O_{\mathbb P^3}(n-1)\otimes \mathcal I_N)=h^0(S, \mathcal O_{S}(n-1)\otimes \mathcal I_N)=\\
&={{n+3}\choose 3}-\delta-1-\frac {n(n+3)}2={{n+2}\choose 3}-\delta
\end{eqnarray*}
hence if and only if $N$ imposes independent conditions to $|\mathcal O_{\mathbb P^3}(n-1)|$, so if and only if $V$ is very regular. 
 \end{proof}

 \begin{remark}\label{rm:regularity} The regularity hypothesis in 
 Proposition \ref{prop:good} is verified for every irreducible component of $V^{\mathbb P^3, |\mathcal O_{\mathbb P^3}(n)|}_\delta$ if either $n\leq 7$ or $n\geq 8$ and $\delta\leq 4n-5$ (see  \cite[Thm. 1.1]{Kl}).
 \end{remark}

In view of the applications of the above results, we record the following  proposition.

\begin{proposition}\label{prop:goodab}
Every non-empty irreducible component of $V^{\mathbb P^3, |\mathcal O_{\mathbb P^3}(n)|}_\delta$ is very regular, hence good, if $n\leq 5$.
\end{proposition}

\begin{proof} The result follows by Proposition \ref{prop:good} and by   \cite[Cor. 1.3]{D}\end{proof}

\begin{remark}\label{rem:5ic} (a) We recall that the maximal number of nodes of a reduced surface of degree $n\leq 5$ in $\mathbb P^3$  is $4$ for a cubic surface, $16$ for a quartic surface, and $31$ for a quintic surface (see, as a general reference for nodal surfaces in $\mathbb P^3$, Lab's thesis \cite {Labs}). Four nodal cubic surfaces are known as \emph{Cayley cubics},  $16$-nodal quartic surfaces are known as \emph{Kummer surfaces} and $31$-nodal quintic surfaces are known as \emph{Togliatti surfaces}. The Severi varieties $V^{\mathbb P^3, |\mathcal O_{\mathbb P^3}(3)|}_4$ and  $V^{\mathbb P^3, |\mathcal O_{\mathbb P^3}(4)|}_{16}$ are also known to be irreducible.

Notice that for the variety of Kummer surfaces  $V^{\mathbb P^3, |\mathcal O_{\mathbb P^3}(4)|}_{16}$ and for the variety of Togliatti surfaces $V^{\mathbb P^3, |\mathcal O_{\mathbb P^3}(5)|}_{31}$ the equality holds in \eqref {eq:lim}. \smallskip

 (b)  The maximum number of nodes of a sextic surface in $\mathbb P^3$ is 65 (see, e.g., \cite {Labs}). A component of the Severi variety $V^{\mathbb P^3, |\mathcal O_{\mathbb P^3}(6)|}_{65}$  is regular by Remark 2.5, but is not very regular, because for it the bound in \eqref {eq:lim} fails. It is a problem to see if there is a very regular, hence good, component of $V^{\mathbb P^3, |\mathcal O_{\mathbb P^3}(6)|}_{\delta}$ for $\delta\leq \delta_0(6)=52$. More generally it is not known which is the maximal value of $\delta$ such that there exists a very regular, hence good,  component of  $V^{\mathbb P^3, |\mathcal O_{\mathbb P^3}(n)|}_{\delta}$, with $n\geq 6$. 
 In particular, do there exist, for $n\geq 4$, very regular components of  $V^{\mathbb P^3, |\mathcal O_{\mathbb P^3}(n)|}_{\delta_0(n)}$, where $\delta_0(n)$ is defined in \eqref{eq:lim}?
  \smallskip

 (c)  In \cite [Thm. 4.6]{CG} we proved that  there is an irreducible, regular component of 
 $V^{\mathbb P^3, |\mathcal O_{\mathbb P^3}(n)|}_{\delta}$, for any $\delta \leq {{n-1}\choose 2}$. It is easy to check that in fact we prove there more than stated, since we prove that there is an irreducible, very regular component of 
 $V^{\mathbb P^3, |\mathcal O_{\mathbb P^3}(n)|}_{\delta}$, for any $\delta \leq {{n-1}\choose 2}$. 
\end{remark}

\section{Deformations of surfaces with an ordinary singularity}\label{sec:ord}

In this section we consider Problem \ref {prob:main} for a  reduced surface $S$ of degree $n$ in $\mathbb P^3$  with a unique \emph{ordinary singularity} in  $p$ of multiplicity $m$.  Recall that $S$ has a point of multiplicity $m$ at $p$ if, given  the minimal desingularization $f: S'\longrightarrow S$ of $S$ at $p$, then,  the exceptional divisor of $S'$ over $p$ is a plane curve $C$ of degree $m$.

\begin{definition}\label{def:ordinary}
With the above notation, if the curve  $C$ is  smooth and $S'$ can be obtained from $S$ by blowing up only once at $p$, then one says that the singularity of $S$ at $p$ is ordinary.  If moreover $C$  is a \emph{general} plane curve of degree $m$ we will say that $p$ is a  \emph{general ordinary multiple} point. If $C$ is reduced and nodal and $S'$ can be obtained from $S$ by blowing up only once at $p$, we will say that $p$ is a \emph{quasi--ordinary multiple} point. 
\end{definition}

We note that the family $\mathcal T_{n,m}$ of surfaces of degree $n$ in $\mathbb P^3$ with a point of multiplicity $m$ is irreducible of dimension

\begin{equation}\label{codimension}
t_{n,m}={{n+3}\choose 3}-{{m+2}\choose 3}+2 =\dim(|\mathcal O_{\mathbb P^3}(n)|)-\delta_0(m)-1
\end{equation}
(where $\delta_0(m)$ is defined in \eqref{eq:lim})  and the general point of $\mathcal T_{n,m}$ corresponds to a surface with a general ordinary singularity  of multiplicity $m$.  

In order to attack Problem \ref {prob:main} in the case under consideration,  we will use a deformation argument that we are going to describe now. 

\subsection{The deformation}\label{ssec:def}  Let $\mathcal X'=\mathbb P^3\times \mathbb D\longrightarrow \mathbb D$ be a trivial family parametrized by the disc $\mathbb D=\{t\in\mathbb C|\,|t|<1\}$,
with fibres $\mathcal X'_t\simeq \mathbb P^3$ over $t\in\mathbb D$.

Next we blow-up the point $p\in \mathcal X'_0$.
Denote by $\mathcal  X\longrightarrow \mathcal X'$  the blowing-up of $\mathcal X'$ at $p\in \mathcal X'_0$. Then $\mathcal X$ has fibre $\mathcal X_t\simeq\mathbb P^3$ if $t\neq 0$ and central fibre is
 $$\mathcal X_0=A\cup\Theta,$$
 where $A$ is the blowing-up  $p: A\longrightarrow \mathbb P^3$ of $\mathbb  P^3$ at $p$ and $\Theta\cong \mathbb P^3$ is the exceptional divisor, intersecting transversally
 $A$ along the plane $E=A\cap \Theta$.

Let $S\subset\mathcal X'_0\cong \mathbb P^3$ be a degree $n$ surface with an ordinary or quasi--ordinary point of multiplicity $m$ at $p$ and at most nodes as further singularities. We  denote by $H$ the Cartier divisor on $\mathcal X$ which is pull-back of a general plane divisor of $\mathbb P^3$ with respect to the projection morphism $\mathcal X\longrightarrow \mathbb P^3$.  The strict transform $S_A$ of  $S\subset\mathcal X'_0$ on $\mathcal X$ is nothing but the minimal resolution $S'$ of $S$ at $p$, it is contained in the component $A$ of $\mathcal X_0$ and it belongs to the linear system $|\mathcal O_A(nH-mE)|$. Under our hypotheses on $S$, the intersection $C=S_A\cap E$ is a reduced plane curve of degree $m$ in $E\cong \mathbb P^2$ with at most nodes.

Let now $S_0:=S_A\cup S_\Theta\in |\mathcal O_{\mathcal X_0} (nH-m\Theta)|$, be  a Cartier divisor with 
$S_\Theta\subset \Theta$. Then,  observing that $\Theta\cong \mathbb P^3$ and $\Theta|_\Theta:=H_\Theta$ is a plane, we have that  $S_\Theta\in |mH_\Theta|$ is a surface of degree $m$ in $\Theta\cong \mathbb P^3$ intersecting the plane $E$ along the curve $C$.  We require that $S_\Theta$  is smooth along $C$, so that  $S_\Theta$  has at most isolated singularities off $E$. 

Note that $\mathcal O_{\mathcal X}(nH-m\Theta)$ is a line bundle on $\mathcal X$ such that its restriction to $\mathcal X_t$, for $t\neq 0$, is $\mathcal O_{\mathcal X_t}(nH)\cong \mathcal O_{\mathbb P^3}(n)$.

\subsection{Nodal deformations}

In what follows we will assume that $S_\Theta$ is nodal and we will 
be interested in nodal deformations of a surface of the type of $S_0$ in $\mathcal X$. 

\begin{proposition}\label{prop:frt} In the above set--up, suppose that $C$ has $\delta_C$ nodes, that $S_A$ has $\delta_A$ nodes and that $S_\Theta$  has $\delta_\Theta$ nodes all off $E$. Let us denote by $N$ the set of these $\delta:=\delta_A+\delta_C+\delta_\Theta$ nodes. Suppose  that 
$h^1(\mathcal X_0, \mathcal O_{\mathcal X_0} (nH-m\Theta)\otimes I_N)=0$, 
that is equivalent to say that the points in $N$ give independent conditions to surfaces in $|\mathcal O_{\mathcal X_0} (nH-m\Theta)|$.  Then $S_0$ can be deformed to a $\delta$--nodal surface $S_t\subset \mathcal X_t\cong \mathbb P^3$ of degree $n$, for $t\neq 0$ and the point corresponding to $S$ is contained in  the closure of a regular component of the Severi variety $V^{\mathbb P^3, |\mathcal O_{\mathbb P^3}(n)|}_\delta$.
\end{proposition}

\begin{proof} The proposition follows as an immediate application of \cite [Thm. 3.13]{CG}.
\end{proof}

\subsection{Applications} 

\subsubsection{The case of general ordinary multiple point} Our first application is to the case $p$ is a general ordinary multiple point of multiplicity $m$ for $S$. 

\begin{theorem}\label{thm:ord} Let $\delta$ be such that there is a very regular irreducible component $V$ of the Severi variety $V^{\mathbb P^3, |\mathcal O_{\mathbb P^3}(m)|}_\delta$. Then the variety $\mathcal T_{n,m}$  is contained  in the Zariski closure of a regular component of $V^{\mathbb P^3, |\mathcal O_{\mathbb P^3}(n)|}_{\delta}$. In other words, the general degree $n$ surface $S$ with a general ordinary singularity of multiplicity $m$ is the limit of surfaces of degree $n$ with $\delta$ nodes tending to the multiplicity $m$ point. 
\end{theorem} 

\begin{proof} Let $[S]\in \mathcal T_{n,m}$ be  a general point. The corresponding surface $S\subset\mathbb P^3$ has a unique multiple point that is a general ordinary singularity of multiplicity $m$. 
If we perform the construction we made in Section \ref {ssec:def},   the curve $C$ cut out by the strict transform $S_A$ of $S$ on $E=A\cap \Theta$ is a general plane curve of degree $m$. Now, if $\delta$ and $V$ are as in the statement,  we can find $S_\Theta\in |mH_{\Theta}|=|\mathcal O_{\mathbb P^3}(m)|$ cutting $C$ on $E$,   with $\delta$ nodes lying in $V$ because $V$ is good by Proposition \ref {prop:good}. Moreover, since $V$ is also regular,  the nodes of $S_\Theta$ impose independent conditions to the surfaces in $|\mathcal O_{\mathbb P^3}(m)|$. Hence the assertion follows by Proposition \ref {prop:frt}. 
\end{proof} 

In order to prove a corollary of this theorem, we need some  preliminary results. 

\begin{proposition}\label{prop:triple} Let $(S,p)$ be a germ of surface singularity in $(\mathbb C^3, \bf 0)$, with $p$ an isolated point of multiplicity 3 with tangent cone $\Gamma$ with vertex $p$ over a smooth curve $C$ of degree 3 in $\mathbb P^2$. Then $(S,p)$ is analitically equivalent to $(\Gamma, p)$.
\end{proposition}

\begin{proof} Suppose that $\Gamma$ has equation $f(x,y,z)=0$, where  $f(x,y,z)=0$ is a homogeneous polynomial of degree 3 that defines $C$ in $\mathbb P^2$. By the \emph{finite determinacy theorem} \cite [Thm. 2.23]{GLS}, $(S,p)$ is analitically isomorphic to $(\Gamma, p)$ if 
\begin{equation}\label{eq:det}
\mathfrak m^4\subseteq \mathfrak m^2 \cdot \Big ( \frac {\partial f}{\partial x}, \frac {\partial f}{\partial y}, \frac {\partial f}{\partial z}\Big )
\end{equation}
where $\mathfrak m$ is the maximal ideal $(x,y,z)$ in $\mathbb C[x,y,z]$. Let us prove that \eqref {eq:det} holds with in fact the equality. 
Indeed, consider the multiplication map
$$
\mu: \mathbb C[x,y,z]_2  \otimes \Big ( \frac {\partial f}{\partial x}, \frac {\partial f}{\partial y}, \frac {\partial f}{\partial z}\Big )\longrightarrow \mathbb C[x,y,z]_4,
$$
where $\mathbb C[x,y,z]_d$ is the vector space of homogeneous polynomials of degree $d$. Since $C$ is smooth, $\Big (\frac {\partial f}{\partial x}, \frac {\partial f}{\partial y}, \frac {\partial f}{\partial z}\Big)$ is a regular sequence, hence 
the kernel of $\mu$ has exactly dimension 3, being generated by the Koszul syzygies of $\frac {\partial f}{\partial x}, \frac {\partial f}{\partial y}, \frac {\partial f}{\partial z}$. Thus 
$$
\dim ({\rm Im}(\mu))=3 \dim (\mathbb C[x,y,z]_2 )-3=15=\dim (\mathbb C[x,y,z]_4)
$$
that implies the assertion.\end{proof}

\begin{lemma}\label{lem:onto}  Let $S$ be a cubic surface in $\mathbb P^3$ with an isolated singularity at  $p$. Then the natural map
$$
H^0(S, N_{S| \mathbb P^3})\longrightarrow H^0(T^1_{S,p})
$$
is surjective. In particular, all local deformations of the singularity $(S,p)$ can be realised by deforming $S$ as a cubic in $\mathbb P^3$.
\end{lemma}

\begin{proof} We have the long exact sequence of sheaves
$$
0\longrightarrow T_S\longrightarrow T_{\mathbb P^3|_S}\longrightarrow N_{S| \mathbb P^3}\longrightarrow T^1_{S,p}\longrightarrow 0.
$$
We denote by $N'_{S| \mathbb P^3}$ the kernel of the  map $N_{S| \mathbb P^3}\longrightarrow T^1_{S,p}$, so that we have the two short exact sequences
\begin{equation}\label{eq:uno}
0\longrightarrow T_S\longrightarrow T_{\mathbb P^3|_S}\longrightarrow N'_{S| \mathbb P^3}\longrightarrow 0
\end{equation}
and
\begin{equation}\label{eq:due}
0\longrightarrow N'_{S| \mathbb P^3}\longrightarrow N_{S| \mathbb P^3}\longrightarrow T^1_{S,p}\longrightarrow 0.
\end{equation}

We claim that $h^2(S, T_S)=0$. Indeed by Serre duality $h^2(S, T_S)=h^0(S, \Omega_S(-1))$. So it suffices to prove that $h^0(S, \Omega_S)=0$.

We have the exact sequence
$$
0\longrightarrow \mathcal O_{S}(-3)\longrightarrow \Omega_{\mathbb P^3|_S}\longrightarrow \Omega_{S}\longrightarrow 0.
$$
By the vanishing of $h^1(S,\mathcal O_{S}(-3))=h^1(S,\mathcal O_{S}(2))=0$, the assertion will follow by proving that $h^0(S,\Omega_{\mathbb P^3|_S})=0$. 
We look at the exact sequence
\begin{equation}\label{eq:bott}
0\longrightarrow \Omega_{\mathbb P^3}(-3)\longrightarrow \Omega_{\mathbb P^3}\longrightarrow \Omega_{\mathbb P^3|_S}\longrightarrow 0.
\end{equation}
We have clearly $h^0(\mathbb P^3,  \Omega_{\mathbb P^3})=0$ and also $h^1(\mathbb P^3,  \Omega_{\mathbb P^3}(-3))=0$ by Bott's formula. Hence $h^0(S,\Omega_{\mathbb P^3|_S})=0$. 

We claim next that $h^1(S, T_{\mathbb P^3|_S})=0$. In fact we have the Euler sequence
$$
0\longrightarrow \mathcal O_S\longrightarrow \mathcal O_S(1)^{\oplus 4}\longrightarrow T_{\mathbb P^3|_S} \longrightarrow 0.
$$
Since $h^1(S,\mathcal O_S(1))=0$ and $h^2(S,\mathcal O_S)=h^0(S,\omega_S)=h^0(S,\mathcal O_S(-1))=0$, the assertion follows. 

Then, since $h^1(S, T_{\mathbb P^3|_S})=h^2(S, T_S)=0$,   from \eqref {eq:uno} follows that $h^1(S, N'_{S|\mathbb P^3})=0$. Then the assertion follows from \eqref {eq:due}.
\end{proof}

By applying Proposition \ref{prop:goodab}, Remark \ref {rem:5ic}, (c),   Theorem \ref{thm:ord}, Proposition \ref {prop:triple}  and Lemma \ref {lem:onto}, we have:

\begin{corollary}\label{cor:gtp} One has:
\begin{itemize}
\item [(i)] $\mathcal T_{n,3}$ is contained in the closure of $V^{\mathbb P^3, |\mathcal O_{\mathbb P^3}(n)|}_{4}$ as a subvariety of codimension 3, but not in the closure of  $V^{\mathbb P^3, |\mathcal O_{\mathbb P^3}(n)|}_{\delta}$ with $\delta>4$. More precisley, an isolated ordinary triple point of a surface in $\mathbb P^3$ is never a limit of more than four nodes; 
\item[(ii)] $\mathcal T_{n,4}$ is contained in the closure of a regular component of $V^{\mathbb P^3, |\mathcal O_{\mathbb P^3}(n)|}_{16}$,   as a codimension $1$ subvariety; 
\item[(iii)] $\mathcal T_{n,5}$ is contained in the closure of a regular component of $V^{\mathbb P^3, |\mathcal O_{\mathbb P^3}(n)|}_{31}$,  as a codimension $1$ subvariety; 
\item [(iv)] for any $m\geq 5$, $\mathcal T_{n,m}$ is contained in the closure of a regular component of $V^{\mathbb P^3, |\mathcal O_{\mathbb P^3}(n)|}_{{m-1}\choose 2}$.
\end{itemize}\end{corollary}

Note that Proposition \ref {prop:triple} and Lemma \ref {lem:onto} are used only for  (i).

\begin{remark}\label{rem:kl}  Concerning point (i)  of Corollary \ref {cor:gtp}, notice that  $V^{\mathbb P^3, |\mathcal O_{\mathbb P^3}(n)|}_{4}$ is clearly irreducible and regular. By \cite [Thm 1.1]{Kl}, any irreducible component of $V^{\mathbb P^3, |\mathcal O_{\mathbb P^3}(n)|}_{16}$ [resp., of $V^{\mathbb P^3, |\mathcal O_{\mathbb P^3}(n)|}_{31}$] is regular as soon as $n\geq 6$ [resp., as soon as $n\geq 9$]. 

\end{remark}

\subsubsection{The case of a (not necessarily general) ordinary singularity} 

We have: 

\begin{theorem}\label{thm:ord1} Let $S$ be a surface of degree $n$ in $\mathbb P^3$  with a unique {ordinary singularity}  $p$ of multiplicity $m$. Consider the minimal desingularization $f: S'\longrightarrow S$ of $S$ with the exceptional divisor of $S'$ over $p$  a plane curve $C$ of degree $m$. Let $\delta$ be such that there is a surface $X$ of degree $m$ in $\mathbb P^3$ having $C$ as a plane section and belonging to a regular component of 
$V^{\mathbb P^3, |\mathcal O_{\mathbb P^3}(m)|}_{\delta}$. Then the point corresponding to $S$ belongs to the closure of a regular component of $V^{\mathbb P^3, |\mathcal O_{\mathbb P^3}(n)|}_{\delta}$.
\end{theorem} 

\begin{proof} Again this is an immediate consequence of Proposition \ref {prop:frt}.
\end{proof} 

\begin{remark}\label{rem:lip} Let $C$ be a smooth plane curve of degree $m$. Theorem \ref {thm:ord1} poses the problem of determining the maximum $\delta(C)$ of $\delta$ such that there is a surface $X$ of degree $m$ in $\mathbb P^3$ having $C$ as a plane section and belonging to a regular component of 
$V^{\mathbb P^3, |\mathcal O_{\mathbb P^3}(m)|}_{\delta}$.

If $C$ is general, one has $\delta(C)\leq \delta_0(m)$. In fact, by Proposition \ref {prop:good}, a regular and good component of the Severi variety is also very regular, and therefore  the assertion follows by Proposition \ref {prop:sev}. As in Remark \ref {rem:5ic}, (b), one can ask whether $\delta(C)=\delta_0(m)$ if $C$ is general. Suppose that this is the case. Then, by the proof of Theorem \ref {thm:ord1}, it would follow that $\mathcal T_{n,m}$ is contained in the closure of a regular component of $V^{\mathbb P^3, |\mathcal O_{\mathbb P^3}(n)|}_{\delta_0(m)}$ as a divisor. 

 Can the invariant $\delta(C)$ vary with $C$? \end{remark}

\subsection{An interesting example of a quasi--ordinary multiple point} In this section, we want to show that the answer to Problem \ref {prob:main} certainly depends on whether the point $p\in S$ ($S$ as usual of degree $n$ with no other singularity) is ordinary or quasi--ordinary. We make this by considering the example of a triple point. 

If the triple point $p\in S$ is ordinary, then, as we saw, the point corresponding to $S$ is the closure of $V^{\mathbb P^3, |\mathcal O_{\mathbb P^3}(n)|}_{4}$ and in general it does not sit in the closure of $V^{\mathbb P^3, |\mathcal O_{\mathbb P^3}(n)|}_{\delta}$ with $\delta>4$ (see Corollary \ref {cor:gtp}, (i)).  We will now show that there are quasi--ordinary triple points $p\in S$ such that $S$ is the closure of $V^{\mathbb P^3, |\mathcal O_{\mathbb P^3}(n)|}_{7}$.  

\begin{lemma}\label{lem:cubic} Let $X\subset \mathbb P^3$ be a 4--nodal cubic.
There is a plane section of $X$, not containing any of the four nodes, that consists of three lines not in a pencil. 

The seven points consisting of the four nodes and of the three vertices of this triangle give independent conditions to cubics in $\mathbb P^3$.
\end{lemma}

\begin{proof}  It is well known that $X$ is the image of the plane via the rational map determined by the linear system of cubic curves passing through the vertices $P_1,\ldots, P_6$ of a \emph{quadrilateral}. 

\begin{figure}
\includegraphics[width=10 cm]{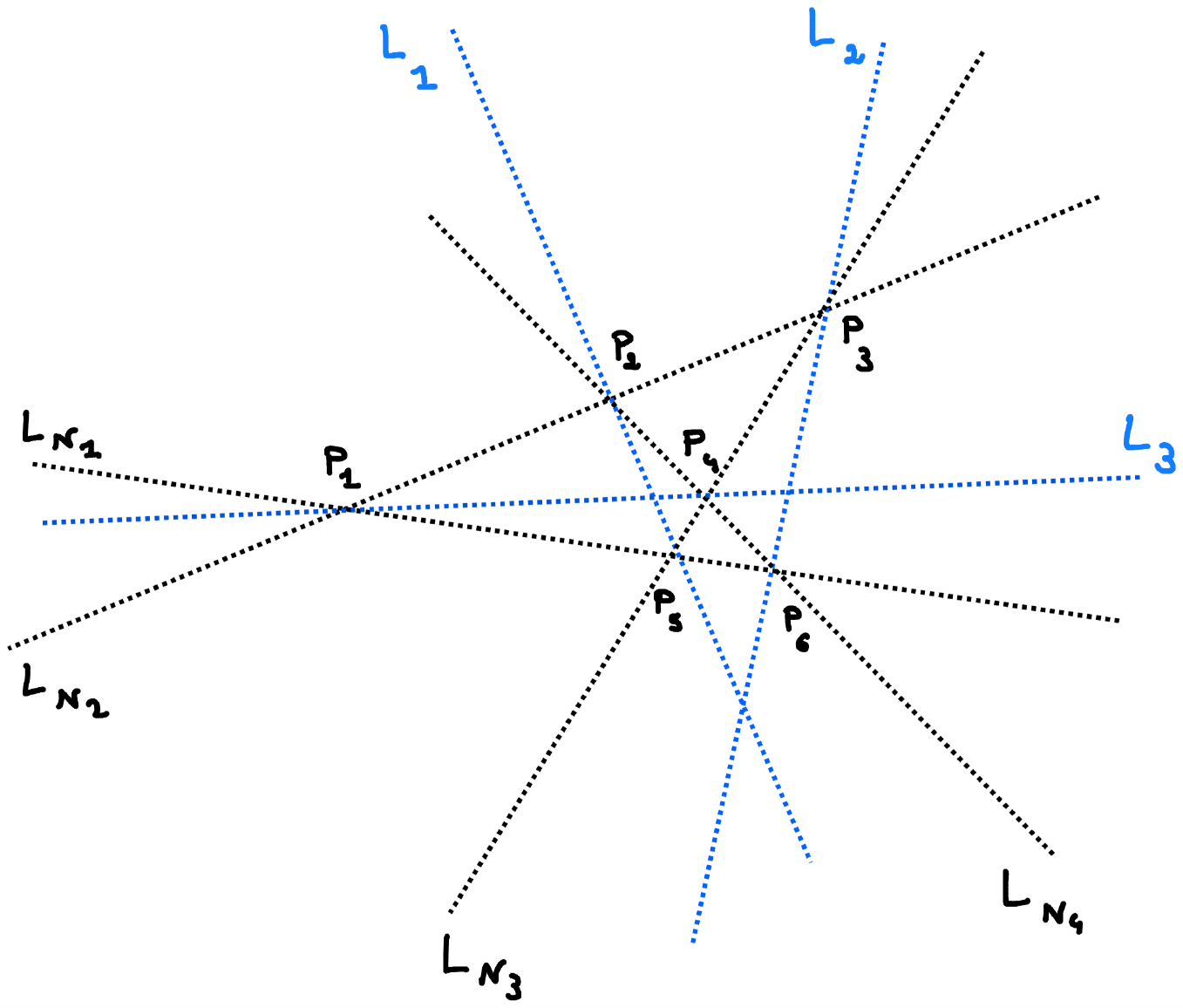}
\caption{}
\end{figure}

In Figure 1 the lines $L_{N_1},\ldots, L_{N_4}$ of the quadrilateral are depicted in black. They are mapped to the four nodes $N_1,\ldots, N_4$ respectively. 

Look at the three blue  lines $L_1,L_2,L_3$ in Figure 1. Each of them is mapped to a line and the union of the three lines forms a cubic through $P_1,\ldots, P_6$. This proves the first assertion of the lemma. 

As for the second assertion, we consider the linear system $\mathcal L$ of cubics passing through the four nodes of $S$. This has dimension $15$ because the nodes impose independent conditions to the cubics. Consider the restriction of the linear system $\mathcal L$ to the plane containing $L_1,L_2,L_3$.  The kernel of this restriction is isomorphic to the linear system of quadrics passing through the four nodes, that has dimension 5. So the restriction linear system has dimension 9 hence it  is the complete linear system of cubics in the plane, and this implies the assertion. \end{proof}

\begin{proposition}\label{prop:7} Consider a surface $S\subset \mathbb P^3$ of degree $n\geq 4$ with, as its only singularity,  a quasi--ordinary triple point $p$, whose minimal resolution $f: S'\longrightarrow S$ is such that  the exceptional divisor $C$ over $p$ is a plane cubic with three nodes, i.e., it is a plane cubic consisting of three lines not in a pencil. Then $S$ is in the closure of the Severi variety $V^{\mathbb P^3, |\mathcal O_{\mathbb P^3}(n)|}_{7}$.
\end{proposition}

\begin{proof} Let us go back to the deformation set--up as in Section \ref {ssec:def}. We have $S_A=S'$, and we may assume that  $S_\Theta$ is the 4--nodal cubic surface having as plane section the triangle $C$, as shown in Lemma \ref {lem:cubic}. Then the assertion follows from Proposition \ref {prop:frt}. \end{proof}

\begin{remark} Again by the results in \cite {Kl}, any irreducible component of $V^{\mathbb P^3, |\mathcal O_{\mathbb P^3}(n)|}_{7}$ is regular. 
\end{remark}

\section{Some rational surfaces in $\mathbb P^3$} \label{sec:rat}

In this section we will introduce and study some irreducible rational surfaces $S$ of degree $n\geq 3$ in $\mathbb P^3$ that have a line $R$ of points of multiplicity $n-2$. These surfaces will be used in Section \ref {sec:line} when we will consider Problem \ref {prob:main} for surfaces with a double line.

\subsection{A remark on hyperelliptic curves} Let  $n\geq 3$  be an integer. Let us consider the linear system $\mathcal C$ of dimension $3n-4$ of plane curves of degree $n-1$ with a fixed  point $q\in \mathbb P^2$ of multiplicity at least $n-3$.   For $n=3$ we are just considering the complete linear system of conics and a fixed point $q$ in the plane.  There is a dense open subset  $U$ of $\mathcal C$ such that all points in $U$ correspond to curves that have an ordinary singularity of multiplicity $n-3$ at $q$ and are elsewhere smooth, so that the normalization of these curves have genus 
$p:=n-3\geq 0$  and have the $g^1_2$ cut out by the lines through $q$. If $G$ is the group of projectivities that fixes $q$, we have an obvious morphism $f: U/G\longrightarrow \mathcal M^1_{p,2}$ where $\mathcal M^1_{p,2}$ is the moduli space of pairs $(C,\mathfrak h)$ with $C\in \mathcal M_p$ a curve of genus $p$ and $\mathfrak h$ a $g^1_2$ on $C$. Note that $\mathcal M^1_{p,2}$ coincides with the locus in $\mathcal M_p$ of hyperelliptic curves, if $p\geq 2$, whereas  $\mathcal M^1_{1,2}$ is the relative Picard bundle of degree 2 line bundles on $\mathcal M_1$, and, finally, $\mathcal M^1_{0,2,}\simeq {\mathbb P^2}^*$ is the set of $g^1_2$ on $\mathbb P^1$. In particular, the dimension of $\mathcal M^1_{p,2}$ is $2p-1$ if $p\geq 2$, and it is $2$ if $p=0,1$.

\begin{proposition}\label{prop:kop} The map $f: U/G\longrightarrow \mathcal M^1_{p,2}$ is dominant.
\end{proposition} 

\begin{proof}Take $(C,\mathfrak h)$ general in $\mathcal M^1_{p,2}$. Consider  a general, non special, $g^1_{n-2}$ on $C$, that we denote by $\mathfrak k$. The minimal sum of $\mathfrak h$ and $\mathfrak k$ has dimension $3$ and maps $C$ isomorphically to a curve of degree $n$ and type $(2,n-2)$ on a smooth quadric in $\mathbb P^3$. Let $x\in C$ be a general point, and project $C$ down from $x$ to a plane $\mathbb P^2$. The image $\Gamma$ of $C$ is a curve of degree $n-1$ having a point of multiplicity $n-3$, as required. \end{proof}

\begin{remark}\label{prop:kop-rm}

Assume $n\geq 5$, i.e., $p\geq 2$. Then the general fibre of $f$ has dimension 
$$
\dim(U/G)-\dim(\mathcal M^1_{p,2})=(3n-10)-(2p-1)=n-3=p,
$$
which means that, given a general  hyperelliptic curve $C$ of genus $p=n-3$, it can be realized as a plane curve of degree $n-1$ with a point of multiplicity $n-3$ and no other singularity  in $\infty^p$ ways, i.e., the degree $n-1$ embedding line bundle in the plane varies in the whole ${\rm Pic}^{n-1}(C)$. Of course the same conclusion
 is trivially true for $n=3$, and it also  holds for $n=4$.  In these two cases, the general fibre of $f$ has dimension
$
\dim(U/G)-\dim(\mathcal M^1_{p,2})=0.
$
\end{remark}

\subsection{Certain linear systems of plane curves}\label{ssec:constr} Let us  consider a plane curve $\Gamma$ of degree $n-1\geq 2$ with an ordinary point $q$ of  multiplicity $n-3$ and no other singularity, with normalization $C\longrightarrow \Gamma$ a curve of genus $p=n-3$ and let us consider the  linear series $\mathfrak h=g^1_2$ on $C$ cut out by the lines through $q$. Suppose that $(C,\mathfrak h)$ is a general point in $\mathcal M^1_{p,2}$. 
 The ramification points of $\mathfrak h$ are $2n-4$. Hence there are $2n-4$ distinct lines through $q$ that are tangent to $\Gamma$ in points that, by generality,  we can assume to be off $q$.

Let us write  $3n-4=2\ell+\epsilon$, with $0\leq \epsilon\leq 1$, and notice that 
$4n-8\geq 2\ell$. Indeed, the assertion is trivial for $n=3$ and if $n\geq 4$ then $4n-8\geq 3n-4\geq 2\ell$.  Thus, we can fix $\ell\leq 2n-4$ of the $2n-4$ tangent lines to $\Gamma$ passing through $q$ and let us denote them by $r_1,\ldots, r_\ell$.
Let us denote by $q_1,\ldots, q_\ell$ the tangency points of $r_1,\ldots, r_\ell$ with $\Gamma$ off $q$. Let us also denote by $q'_i$ the infinitely near point to $q_i$ along $r_i$, for $i=1,\ldots, \ell$, and notice that the $2\ell$ points $q_i, q'_i$, for $i=1,\ldots, \ell$, belong  to $\Gamma$. If $\epsilon=1$ let us  fix also a general point $\bar q$ on $\Gamma$. 

Consider now the linear system $\mathcal L$ of plane curves of degree $n$ that have at $q$ a point of multiplicity at least $n-2$ and pass through the points $q_i, q'_i$,  for $i=1,\ldots, \ell$, and through $\bar q$ if $\epsilon=1$. 

We want to study the system $\mathcal L$ and the image of the plane via the rational map determined by $\mathcal L$. To do this, it is convenient to pass to the blow--up $f: X\longrightarrow \mathbb P^2$ of the plane at the point $q$, at the points $q_i$ and then at the infinitely near points $q'_i$, for $i=1,\ldots, \ell$, and at $\bar q$ if $\epsilon=1$. We let $\mathcal L'$ be the line bundle on $X$ corresponding to the strict transform of $\mathcal L$ on $X$. Let $R$ be the pull--back on $X$ of a general line of $\mathbb P^2$, let $L$ be the strict transform on $X$ of a general line of $\mathbb P^2$ through $q$ and we denote by $C$  the strict transform on $X$ of the curve $\Gamma$, which is the normalization of $\Gamma$. Then we have $\mathcal O_C(L)\sim \mathfrak h$ and, as we observed, $\mathcal O_C(R)$ is a general line bundle on $C$ of degree $n-1$.

Consider the exact sequence
\begin{equation}\label{eq:lp}
0\longrightarrow \mathcal L'(-C)\longrightarrow \mathcal L'\longrightarrow \mathcal L'_{|C}\longrightarrow 0.
\end{equation}
We notice that $\mathcal L'(-C)\cong \mathcal O_X(L)$, and therefore $h^1(X, \mathcal L'(-C))=0$. This implies that $|\mathcal L'|$ cuts out on $C$ a complete linear series $\mathfrak g$. The degree of $\mathfrak g$ is 
$$
n(n-1)-(n-2)(n-3)-(3n-4)=n-2= p+1
$$
and more precisely one has that
$$
\mathfrak g\sim (n-2)\mathfrak h +2R_{|C}- \ell \mathfrak h-\epsilon\bar q
$$
and, as $\mathcal O_C(R)$ varies in ${\rm Pic}^{n-1}(C)$, we see that $\mathfrak g$ varies in the whole of ${\rm Pic}^{n-2}(C)$, so that with an appropriate choice of $\mathcal O_C(R)$ we may assume that $\mathfrak g$ is any given general $g^1_{n-2}$ on $C$. 

\subsection{The image surface}\label{subsec: image surface} Looking  at the sequence \eqref {eq:lp}, under the above generality assumptions on $(C, \mathfrak h)$ and on the series $\mathfrak g$, we see that $\dim(|\mathcal L'|)=3$, hence also $\dim (\mathcal L)=3$. 

\begin{lemma}\label{lem:lopg} Under the above generality assumptions on $(C, \mathfrak h)$ and on the series $\mathfrak g$,
the linear system $|\mathcal L'|$ has no base points.
\end{lemma}

\begin{proof} The linear system $|\mathcal L'|$ contains the subsystem consisting of the fixed curve $C$ plus the base point free pencil $|L|$. So the only base points of $|\mathcal L'|$ may occur on $C$. However, as we saw, $|\mathcal L'|$ cuts out on $C$ the complete, base point free linear series $\mathfrak g$, so $|\mathcal L'|$ has no base points. \end{proof}

\begin{proposition}\label{prop:lopg1} Under the above generality assumptions on $(C,\mathfrak h)$ and on the series $\mathfrak g$, the linear system $|\mathcal L'|$ determines a birational morphism of $X$ to a surface $\Sigma$ in $\mathbb P^3$ of degree $n$ with a line $r$, the image of the curve $C$,  of multiplicity $n-2$. 
Via this morphism, the linear system $|L|$ is mapped to the linear system of conics cut out on $\Sigma$ by the planes through the line $r$. 
\end{proposition}

\begin{proof} Let $\varphi: X\longrightarrow \Sigma$ be the morphism determined by $|\mathcal L'|$. Since $|\mathcal L'|$ contains the subsystem consisting of the fixed curve $C$ plus the base point free pencil $|L|$, to see that $\varphi$ is birational onto its image, it suffices to prove that given the general curve $L\in |L|$, $|\mathcal L'|$ cuts out on it a complete linear series of degree 2. To see this, look at the sequence
\begin{equation}\label{eq:retta-generica}
0\longrightarrow \mathcal L'(-L)\longrightarrow \mathcal L'\longrightarrow \mathcal L'_{|L}\longrightarrow 0,
\end{equation}
with $\mathcal L'(-L)=\mathcal O_X(C)$. Then the assertion will be proved if we show that $h^1(\mathcal O_X(C))=0$ or equivalently, by Riemann--Roch theorem, that $h^0(\mathcal O_X(C))=1$. Look at the sequence
$$
0\longrightarrow \mathcal O_X\longrightarrow \mathcal O_X(C)\longrightarrow \mathcal O_{C}(C)\longrightarrow 0.
$$
To prove that $h^0(\mathcal O_X(C))=1$ one has to show that $h^0(\mathcal O_C(C))=0$. 
Note that $C^2=n-4$ and more precisely 
$$
\mathcal O_C(C)\sim(n-3-\ell)\mathfrak h+2R_{|C}-\epsilon\bar q.
$$
Using again the generality of $R_{|C}$, this is a general line bundle of degree $n-4=p-1$ on $C$, and therefore it is not effective, as required. 

The degree of the image surface $\Sigma$ is 
$$
\deg (\Sigma)=(C+L)^2=(n-4)+2\cdot 2=n.
$$
Finally, recalling that
$\mathcal L'_{|C}=\mathfrak g$ is a $g^1_{n-2}$ on $C$ and that $(C+L)\cdot L=C\cdot L=2$, we see that $C$ is mapped by $\varphi$ to a line $r$
 contained in $\Sigma$ with multiplicity $n-2$, 
and the linear system $|L|$ is mapped by $\varphi$ to the linear system of conics in $\Sigma$ cut out by the planes through $r$. 
\end{proof}

\subsection{Nodes} Finally we show that:

\begin{proposition}\label{prop:ojo} The surface $\Sigma$ has $2\ell$ nodes as singularities off the line $r$ of multiplicity $n-2$ that
lie pairwise on $\ell$ disjoint lines contained in $\Sigma$ and meeting $r$ at $\ell$ distinct points. 
These nodes of $\Sigma$ impose independent conditions to the linear system of surfaces of degree $n$ that contain $r$ with multiplicity $n-2$. 

\end{proposition}

\begin{proof} To prove that $\Sigma$ has $2\ell$ nodes off the line $r$ of multiplicity $n-2$, we will prove that on $X$ there are  $(-2)$--curves $D_1,\ldots, D_{2\ell}$, such that $\mathcal L'\cdot D_i=C\cdot D_i=0$, for $i=1,\ldots, 2\ell$, and $D_i\cdot D_j=0$ for $1\leq i<j\leq 2\ell$. To see this, 
notice that the blow--up of the points $q_1,\ldots, q_\ell$, produces $(-1)$--curves 
$E_1,\ldots, E_\ell$. When we blow--up the infinitely near points $q'_1,\ldots, q'_\ell$, with $q'_i\in E_i$, for $i=1,\ldots, \ell$, introducing new exceptional divisors $F_1,\ldots, F_\ell$, the proper transforms of $E_1,\ldots, E_\ell$ on $X$ become $(-2)$--curves $D_1,\ldots, D_{\ell}$, such that $\mathcal L'\cdot D_i=C\cdot D_i=0$ for $i=1,\ldots, \ell$, and $D_i\cdot D_j=0$ for $1\leq i<j\leq \ell$. 
On the other hand, after subsequently blowing up $q$, the proper transforms $D_{\ell+1},\ldots, D_{2\ell}$ of the lines $r_1,\ldots, r_\ell$ become also $(-2)$--curves, such that $\mathcal L'\cdot D_i=C\cdot D_i=0$ for $i=\ell+1,\ldots, 2\ell$, and $D_j\cdot D_j=0$ for $1\leq i<j\leq 2\ell$. 

To see that the  $2\ell$ nodes of $\Sigma$ obtained by the contraction of the $(-2)$--curves $D_1,\ldots, D_{2\ell}$ are pairwise contained into $\ell$ disjoint lines contained in $\Sigma$ and meeting $r$ at $\ell$ distinct points, observe that the $(-1)$--curves $F_1,\ldots, F_\ell$ are such that $F_i\cdot F_j=0$ if $i\neq j$, $F_i\cdot D_i=F_i\cdot D_{i+l}=1$, $F_i\cdot D_j=0$ for $j\neq i, i+l$ and $F_i\cdot \mathcal L'=C\cdot F_i=1.$ 
Thus their image curves in $\mathbb P^3$ by the morphism defined by $\mathcal L'$ are $\ell$ disjoint lines $s_1,\ldots, s_{\ell}$
with the desired properties.

We observe that, if $\epsilon=1$ the blowing-up at the point $\bar q$ produces two further $(-1)$-curves on $X$, the corresponding exceptional divisor $\bar E$ and the proper transform of the line $\bar r$ passing through $q$ and $\bar q$. One has 
$$
\bar E\cdot \bar r=\mathcal L'\cdot \bar E=\mathcal L'\cdot \bar r=C\cdot \bar E=C\cdot \bar r=1.
$$ 
This implies that the morphism defined by $\mathcal L'$ maps isomorphically both $\bar E$ and $\bar r$ to lines in $\Sigma\subset\mathbb P^3$,
which we denote respectively by $\bar e$ and $\bar s$ . We claim that $\bar e$ and $\bar s$
are two distinct lines, generating a plane $\bar \pi$ passing through the multiple line $r$ of $\Sigma$ and tangent to $\Sigma$ at the point  $\bar e\cap\bar s$,
which is a smooth point for $\Sigma$.
Indeed these two lines have to meet at the image of the intersection point of $\bar E$ and $\bar r$ and they meet the multiple
line $r$ of $\Sigma$ at two distinct points, because, by the generality of $\bar q$, we may assume that the two intersection points of the line $\bar r$
with $C$ off the point $q$ are not in a divisor of $|\mathcal L'|_C|$ (which we know to be a complete general $g^1_{n-2}$ on $C$). 

To show that $\Sigma$ has no further singular points outside the line $r$, we first observe that $\Sigma$ may have only singular points of multiplicity $2$ off $r$.
This is because every plane $\pi$ passing through $r$ intersects $\Sigma$ along $r$, with multiplicity $n-2$, and a conic. If  
$\pi_1,\ldots, \pi_\ell$ are the planes generated respectively by the line $r$ and the line $s_i$, with $i=1,\dots, \ell$, then every $\pi_i$ intersects $\Sigma$
along the line $r$, with multiplicity $n-2$ and the line $s_i$, with multiplicity $2$, with $i=1,\dots, \ell$. 

By Lemma \ref{lem:lopg} and by the exact sequence \eqref{eq:retta-generica}, under our generality assumptions, for any  line $l\in|L|$, different from $r_1,\ldots, r_\ell$, and from $\bar r$ if $\epsilon=1$, the linear system
$|\mathcal L'|_l|$ is a complete, base point free linear system of dimension and degree $2$.
Then the morphism defined by $\mathcal L'$ isomorphically maps $l$  to a smooth conic.
This implies that any plane $\pi$ in $\mathbb P^3$ passing through  $r$
other than $\pi_1,\ldots, \pi_\ell$, and than $\bar\pi$ if $\epsilon=1$, intersects $\Sigma$ outside $r$ along a smooth conic.
This proves that $\Sigma$ is smooth outside $r$ and the $2\ell$ aforementioned nodes.

We are now left to prove the final assertion. For this, we must prove that the $(-2)$--curves $D_1,\ldots, D_{2\ell}$ give $2\ell$ independent conditions to the pull--back $\mathcal M$ on $X$ via $\varphi$ of the linear system of surfaces of degree $n$  in $\mathbb P^3$ that contain $r$ with multiplicity $n-2$. One has
$$
\mathcal M\sim n\mathcal L' - (n-2)C\sim 2\mathcal L'+ (n-2)L
$$
that corresponds to the linear system of plane curves of degree $3n-2$ with the point $q$ of multiplicity $3n-6$, with double points at the points $q_1,\ldots, q_\ell$ and $q'_1,\ldots, q'_\ell$, and a further double point at $\bar q$ if $\epsilon =1$. 

\begin{claim}\label{cl:21} The linear system $\mathcal M$ has dimension $6n-4$.
\end{claim}

\begin{proof}[Proof of Claim \ref {cl:21}] First  we estimate the dimension of the linear system $\mathcal S$ of surfaces of degree $n$ in $\mathbb P^3$ having a line $r$ of multiplicity (at least) $n-2$. Take a general plane $\pi$ containing $r$. The system $\mathcal S$ cuts out on $\pi$, off $r$, a linear system of conics. Hence we can split $\pi$ off $\mathcal S$ with at most $6$ conditions. Repeating this $n-2$ times, the residual system is the system of all quadrics, which has dimension 9. Hence we have
$$
\dim(\mathcal S)\leq 6(n-2)+ 9=6n-3. 
$$
Accordingly, we must have $\dim (\mathcal M)\leq 6n-4$. On the other hand 
$$
\dim(\mathcal M)\geq \frac {(3n-2)(3n+1)}2-\frac {(3n-6)(3n-5)}2-6\ell-3\epsilon=6n-4,
$$
hence $\dim(\mathcal M)=6n-4$, as required.
\end{proof}

Now subtract from $\mathcal M$ the $\ell$ lines $r_1,\ldots, r_\ell$. We are left with a linear system $\mathcal N$ and 
$$
\dim (\mathcal N)\geq \dim(\mathcal M)-\ell=6n-4-\ell.
$$

The linear system $\mathcal N$ corresponds to the linear system of plane curves of degree 
$3n-2-\ell$ with the point $q$ of multiplicity $3n-6-\ell$, with simple points at the points $q_1,\ldots, q_\ell$ and $q'_1,\ldots, q'_\ell$, and a  double point at $\bar q$ if $\epsilon =1$. Now we want to split from $\mathcal N$ also the remaining $(-2)$-curves, and this gives us the linear system $\mathcal P$ that corresponds to the linear system of plane curves of degree  $3n-2-\ell$ with the point $q$ of multiplicity $3n-6-\ell$, with double points at $q_1,\ldots, q_\ell$ and a  double point at $\bar q$ if $\epsilon =1$. We have
\begin{equation}\label{eq:fro}
\dim (\mathcal P)\geq \dim (\mathcal N)-\ell\geq \dim(\mathcal M)-2\ell=6n-4-2\ell=3n+\epsilon.
\end{equation}

By the generality assumptions, we may assume that the points $q_1,\ldots, q_\ell$ and $\bar q$ (if $\epsilon=1)$ are general points in the plane. 

Consider now the surface  $Y\subset \mathbb P^r$, with $r=15n-5\ell-16$, that is the image of the plane via the linear system of curves of degree $3n-2-\ell$ with the point $q$ of multiplicity $3n-6-\ell$. 

\begin{claim}\label{cl:nd} The surface  $Y$ is not $(\ell+\epsilon-1)$--defective\footnote{Recall that a (non--degenerate, irreducible, projective) linearly normal surface $S\subset \mathbb P^r$ is said to be $k$--defective if its $k$--secant variety has dimension smaller than expected, or, equivalently, if imposing $k+1$ general double points to $|\mathcal O_S(1)|$ imposes less than $3(k+1)$ independent conditions.}.
\end{claim}

\begin{proof}[Proof of Claim \ref {cl:nd}] The list of $k$--defective surfaces can be found in  \cite [Thm. 1.3] {CC}, where they are divided in two types (i) and (ii). It is clear that  $Y$ does not belong to case (i) for numerical reasons, so we have to exclude that $Y$ falls in case (ii). A $k$--defective, non--degenerate  surface in $\mathbb P^r$  in case (ii)  is contained  in a $(s + 2)$-dimensional cone $W$ over a curve, with vertex a linear space $\Pi$ of dimension $s \leq k - 1$ with $r \geq 2k + s + 3$. We have to exclude that $Y$ is of this type, with $k=\ell+\epsilon-1$.

The surface  $Y$ can be seen as the image of $\mathbb F_1$ via a linear system that we will soon identify. Let us denote by $E$ the $(-1)$--curve of $\mathbb F_1$ and  by $F$ a curve in the ruling of $\mathbb F_1$ over $\mathbb P^1$. The effective cone of $\mathbb F_1$ is spanned by $E$ and $F$. The linear system that maps $\mathbb F_1$ (isomorphically) to $Y$ is 
$|(3n-2-\ell)F+4E|$. One has
$$
((3n-2-\ell)F+4E)\cdot E=3n-6-\ell, \quad ((3n-2-\ell)F+4E)\cdot F=4
$$
so that the minimum degree of a movable curve on $Y$ is $4$ (and it is achieved by the curves in $|F|$). One has 
$$
\deg( Y  )=(3n-2-\ell)^2-(3n-6-\ell)^2=4(3n-4+\epsilon).
$$

Suppose, by contradiction, that $Y$ is of type (ii) as above. Let $\mathfrak R$ be the (possibly empty) 1--dimensional scheme cut out by the vertex $\Pi$ of the cone $W$ on  $Y$, which has degree $\rho\geq 0$. Take a general hyperplane passing through   $\Pi$. Such a hyperplane intersects $W$ in $\Pi$ plus $\gamma:=\deg(W)$ linear spaces of dimension $s+1$ containing $\Pi$, each of which intersects  $Y$, off $\mathfrak R$, in $a\geq 1$ curves of degree $\delta\geq 4$. Looking at the section of such a hyperplane with $Y$, we see that
$$
4(3n-4+\epsilon)=\deg(X)=a\gamma\delta+\rho\geq 4\gamma, \quad \text{hence}\quad 3n-4+\epsilon\geq \gamma.
$$
On the other hand, one has $\gamma\geq r-s-1$, so that we find  
$$
3n-4+\epsilon\geq r-s-1\geq r-k=15n-6\ell-\epsilon-15\geq 6n+2\epsilon-3
$$
that is clearly a contradiction.  \end{proof}

Claim \ref  {cl:nd} implies that  
$$
\dim (\mathcal P)= \frac{(3n-2-\ell)(3n+1-\ell)}2-\frac{(3n-6-\ell)(3n-5-\ell)}2-3\ell -3\epsilon=3n+\epsilon.
$$
Hence we see that in  \eqref {eq:fro} equalities hold everywhere, and this proves that  the $2\ell$ nodes of $\Sigma$  impose independent conditions to the linear system of surfaces of degree $n$ that contain the line $r$ with multiplicity $n-2$.

\end{proof}

\section{Deformations of surfaces with a  double  line}\label{sec:line}

In this section we will consider Problem \ref {prob:main} for a surface $S$ of degree $n\geq 3$ that has only \emph{ordinary singularities} along a line $R$. 
This means that: (a) the points of $R$ are double points for $S$, (b) at the general point $p\in R$, the surface $S$ has local equation $xy=0$, (c) there are finitely many \emph{pinch points} $p\in R$, where the local equation of $S$ is given by $x^2y-z^2=0$.  With the terminology of \cite{FFP} and related references, the surface $S$ is an example of \emph{semi-smooth surface}. In particular, \color{black} the tangent cone to $S$ at the general point $p\in R$ consists of two distinct planes containing $R$, while the tangent cone to $S$ at a pinch point $p\in R$ consists of a double plane. 

A minimal resolution $f: S'\longrightarrow S$ of $S$ is obtained by blowing--up $R$ in $\mathbb P^3$ and by taking the strict transform $S'$ of $S$.

Again we will use a deformation argument that we are going to describe soon. The situation is similar to the one in Section \ref {ssec:def}, but slightly more complicated. 

\subsection{The ambient of the deformation}\label{ssec:lio}

 Let $\mathcal X'=\mathbb P^3\times \mathbb D\longrightarrow \mathbb D$ be a trivial family parametrized by the disc $\mathbb D=\{t\in\mathbb C|\,|t|<1\}$,
with fibres $\mathcal X'_t\simeq \mathbb P^3$ over $t\in\mathbb D$.

Next we blow-up the line $R\subset\mathcal X'_0$.
Denote by $\mathcal  X\longrightarrow\mathcal X'$  the blowing-up of $\mathcal X'$ along $R\subset\mathcal X'_0$. Then $\mathcal X$ has fibre $\mathcal X_t\simeq\mathbb P^3$ if $t\neq 0$ and central fibre 
 $$\mathcal X_0=A\cup\Theta,$$
 where $A$ is the blowing-up  $p: A\longrightarrow \mathbb P^3$ of $\mathbb  P^3$ along $R$ and $\Theta$ is the exceptional divisor, intersecting transversally
 $A$ along the smooth ruled surface
	$$
E=A\cap \Theta\simeq \mathbb P(\mathcal O_{ R}(1) \oplus \mathcal O_{R}(1))\cong\mathbb F_0,
$$
coinciding on $A$ with the exceptional divisor of $p: A\longrightarrow \mathbb P^3$.
 We  have 
 $$
 \Theta\cong \mathbb P(\mathcal O_{R}(1) \oplus \mathcal O_{R}(1) \oplus\mathcal O_{R})\cong 
 \mathbb P(\mathcal O_{R} \oplus \mathcal O_{R} \oplus\mathcal O_{R}(-1)),
$$
which is a $\mathbb P^2$-bundle on $R$ and we denote by $P_\Theta\cong\mathbb P^2$ the linear equivalence class of a
fibre of the natural morphism $\Theta\longrightarrow R$. 

\begin{lemma}\label{lem:axx} There is a natural isomorphism
$$
\Theta \cong {\rm Bl}_{R^*}({\mathbb P^3}^*),
$$
i.e., $\Theta$ can be seen in a natural way as the blowing-up of ${\mathbb P^3}^*$ along the dual line $R^*$ of $R$, namely  $R^*$ is the pencil of planes containing $R$. 
\end{lemma}

\begin{proof} This follows from  \cite[Prop. 9.11]{EH}. However, we need to make the isomorphism $\Theta \cong {\rm Bl}_{R^*}({\mathbb P^3}^*)$ very explicit. 

\begin{claim}\label{cl:1}  The divisor $E$ does not move on $\Theta$.\end{claim}

    \begin{proof}[Proof of  Claim \ref  {cl:1}] Consider the obvious relation
$$
(A+\Theta)\cdot A\cdot\Theta=A^2\cdot \Theta+A\cdot\Theta^2= 0,
$$
where
$A^2\cdot \Theta=c_1(\mathcal N_{E|\Theta})$ and $A\cdot\Theta^2=c_1(\mathcal N_{E|A}):= {\bf e}$. 
 Denote by ${\bf f}_E$ and ${\bf r}_E$ 
the linear equivalence classes, respectively, of a fibre  and a section of $E\cong \mathbb P^1\times \mathbb P^1$ as a $\mathbb P^1$-bundle over $R\simeq \mathbb P^1$,  generating  $\rm{Pic}(E)$ and such that  ${\bf r}_E^2={\bf f}_E^2=0$.

By using that $K_A\sim p^*(K_{\mathbb P^3})+E$ and by the adjunction formula 
$\omega_E\simeq\omega_A\otimes\mathcal N_{E|A}$
one obtains that

$$
\mathcal N_{E|A}\sim \omega_E \otimes 	\omega_{{A}|_E}^{-1}\sim \omega_E\otimes ( p^*(\mathcal O_{\mathbb P^3}(4))
\otimes \mathcal O_A(-E) ) |_E.
$$
Since $\mathcal O_A(E)|_E={\bf e}$, we have
$$
2{\bf e}\sim -2{\bf f}_E-2{\bf r}_E+4{\bf f}_E=2{\bf f}_E-2{\bf r}_E,
$$
hence
$$
c_1(\mathcal N_{E|A})={\bf e}\sim {\bf f}_E-{\bf r}_E\sim -\,c_1(\mathcal N_{E|\Theta}).
$$
This proves that   
$\mathcal N_{E|\Theta}$ is non-effective, as claimed.\end{proof}

\begin{claim}\label{cl:11} 
The linear system $|\mathcal O_{\Theta}(P_\Theta+E)|$ has dimension $3$ and its restriction to 
a plane $P\in |P_\Theta|$ coincides with the complete linear system $|\mathcal O_P(1)|$. Moreover, one has that 
$P_\Theta\cdot(P_\Theta + E)^2=1$, $(P_\Theta + E)|_E\sim {\bf r}_E$ and $E\cdot(P_\Theta + E)^2=0$. \end{claim}

\begin{proof}[Proof of Claim \ref {cl:11}] 

We first prove that $\mathcal O_{P}(P_\Theta+E)\simeq\mathcal O_{P}(E)\simeq \mathcal O_{P}(1).$
By the fact that $\Theta$ is rational, one has $h^1(\mathcal O_\Theta)=h^2(\mathcal O_\Theta)=0$. By the exact sequence
$$
0\longrightarrow \mathcal O_\Theta \longrightarrow \mathcal O_\Theta(P_\Theta)\longrightarrow \mathcal O_{P}(P_\Theta)\cong \mathcal O_{P}\longrightarrow 0,
$$
one sees that $h^1(\mathcal O_\Theta(P_\Theta))=0$. Twisting by $\mathcal O_\Theta(-P_\Theta)$, one also obtains $h^i(\mathcal O_\Theta(-P_\Theta))=0,$
for $0\leq i\leq 3.$ 

Now, a fibre $P\sim P_\Theta$ of $\Theta$ cuts a fibre ${\bf f}_E\cong \mathbb P^1$ on $E$. Thus $P\cap E$ is a smooth rational plane curve. In order to see that
$\mathcal O_{P}(E)\simeq \mathcal O_{P}(1)$,  we prove that $h^0(\mathcal O_{P}(E))=3$.   
Consider the  exact sequences
\begin{equation}\label{eq: 1}
0\longrightarrow \mathcal O_\Theta (E-P_\Theta)\longrightarrow \mathcal O_\Theta(E)\longrightarrow \mathcal O_{P}(E)\longrightarrow 0,
\end{equation}
\begin{equation}\label{eq: 2}
0\longrightarrow \mathcal O_\Theta \longrightarrow \mathcal O_\Theta(E)\longrightarrow \mathcal O_{E}(E)\cong \mathcal O_E({\bf r}_E-{\bf f}_E)\longrightarrow 0,
\end{equation}
and
\begin{equation}\label{eq: 3}
0\longrightarrow \mathcal O_E({\bf r}_E-{\bf f}_E) \longrightarrow \mathcal O_E({\bf r}_E)\longrightarrow \mathcal O_{{\bf f}_E}({\bf r}_E)\cong \mathcal O_{\mathbb P^1}(1)\longrightarrow 0.
\end{equation}
Using \eqref{eq: 3}, 
one obtains $h^1( \mathcal O_E({\bf r}_E-{\bf f}_E))=0$. By \eqref{eq: 2} and by $h^1(\mathcal O_\Theta)=0$, one has 
$h^1( \mathcal O_\Theta(E))=0$. It follows by \eqref{eq: 1} that  
$$
h^0(\mathcal O_{P}(E))=h^0(\mathcal O_{\Theta}(E))+h^1(\mathcal O_\Theta (E-P_\Theta)).
$$
By the exact sequence
$$
0\longrightarrow \mathcal O_\Theta(-P_\Theta) \longrightarrow \mathcal O_\Theta(E-P_\Theta)\longrightarrow \mathcal O_{E}(E-P_\Theta)\cong \mathcal O_E({\bf r}_E-2{\bf f}_E)\longrightarrow 0,
$$
using that 
$h^i(\mathcal O_\Theta(-P_\Theta))=0$, for $0\leq i\leq 3,$  one obtains that
$$
h^1(\mathcal O_\Theta (E-P_\Theta))=h^1(\mathcal O_E (E-P_\Theta))=h^1(\mathcal O_E({\bf r}_E-2{\bf f}_E)).
$$
By Serre duality, $h^1(\mathcal O_E({\bf r}_E-2{\bf f}_E))=h^1(\mathcal O_E(-3{\bf r}_E))=2.$
Thus $h^0(\mathcal O_{P}(E))=3$, hence $\mathcal O_{P}(E)\simeq \mathcal O_{P}(1)$. 
Twisting \eqref{eq: 1} by $\mathcal O_\Theta(P_\Theta)$, one obtains our first statement. 
The equality $P_\Theta\cdot(P_\Theta + E)^2=1$ follows. The relation $(P_\Theta + E)|_E\sim {\bf r}_E$ is trivial, hence $E\cdot(P_\Theta + E)^2=0$,
 as claimed.
 \end{proof}

Now we can finish the proof of the lemma. First  notice that   the linear system $|P_\Theta+E|$ has no base points. Indeed, since it contains the linear system $|P_\Theta|+E$, the base points of $|P_\Theta+E|$ may occur only on $E$. But, by 
the exact sequence
$$
0\longrightarrow \mathcal O_\Theta (P_\Theta)\longrightarrow \mathcal O_\Theta(P_\Theta+E)\longrightarrow \mathcal O_E(P_\Theta+E)\cong \mathcal O_E({\bf r}_E)\longrightarrow 0,
$$
since, as we saw, $h^1(\mathcal O_\Theta (P_\Theta))=0$, $|P_\Theta+E|$ cuts out on $E$ the complete linear system $|{\bf r}_E|$ that is base point free.

From Claim \ref {cl:11} it follows that
$$
\phi_{|P_\Theta+E|}:\Theta\longrightarrow{\mathbb P^3}^*
$$
is a morphism mapping $\Theta$ birationally onto ${\mathbb P^3}^*$. The planes in $|P_\Theta|$ are mapped to planes in the pencil that has as base locus a line $r$ where $E$ is contracted. The line $r$ can be naturally identified with $R^*$. Indeed, $\phi_{|P_\Theta+E|}$ contracts $E\cong \mathbb F_0$ along the ruling $|{\bf  r}_E|$. Hence the image of $E$ can be identified with a fibre ${\bf  f}_E$, which in turn can be identified with $R^*$. \end{proof}

\subsection{Deformation of a surface with a double line}\label{ssec:kip}

Let now $S\subset\mathcal X'_0$ be a degree $n\geq 3$ surface in $\mathbb P^3$ with only ordinary singularities along a line $R$. We  denote by $H$ the Cartier divisor on $\mathcal X$, which is the pull-back of a general plane of $\mathbb P^3$ with respect to the projection morphism $\mathcal X\longrightarrow\mathbb P^3$.  The strict transform $S_A$ of  $S\subset\mathcal X'_0$ on $\mathcal X$ sits on the component $A$ of $\mathcal X_0$, it coincides with the minimal desingularization of $S$ and it belongs to the linear system $|\mathcal O_A(nH-2E)|$.
In particular, since $S$ has only ordinary singularities along $R$, the surface  $S_A$ cuts on $E$ a smooth curve $C$. 

\begin{lemma}\label{lem:sit} The smooth curve $C$ on $E$ is linearly equivalent to 
$$C\sim n{\bf f}_E+2c_1(\mathcal N_{E|\Theta})\sim 2{\bf r}_E+(n-2){\bf f}_E$$
hence it has genus $n-3$. The pencil $|{\bf f}_E|$ cuts out a $g^1_2$ on $C$ whose $2n-4$ ramification points map to the pinch points of $S$ on $R$. 
\end{lemma}

\begin{proof} We consider the line bundle $\mathcal O_\mathcal X(nH-2\Theta)$ whose restriction to the general fibre $\mathcal X_t\cong \mathbb P^3$, for $t\neq 0$, is $\mathcal O_{\mathbb P^3}(n)$. It restricts to $A$ as $\mathcal O_A(nH-2E)$.
One has  
$$( nH-2\Theta)|_\Theta\sim ( nH+2A)|_\Theta\sim n P_\Theta+2 E$$
and therefore
$S_A\cdot E\sim n{\bf f}_E+2c_1(\mathcal N_{E|\Theta})\sim 2{\bf r}_E+(n-2){\bf f}_E$. The rest of the assertion is obvious. \end{proof}

 Let now $S_0:=S_A\cup S_\Theta\in |\mathcal O_{\mathcal X_0} (nH-2\Theta)|$, be  a Cartier divisor with 
$S_\Theta\subset \Theta$. We want to understand what the surface $S_\Theta\subset \Theta$ can be, recalling that $\Theta$ can be seen as the blow--up of ${\mathbb P^3}^*$ along $R^*$ via the morphism determined by the linear system $|P_\Theta+E|$. 

Set  $H_\Theta\sim P_\Theta+E$. Since 
\begin{equation}\label{eq:hki}
( nH-2\Theta)|_\Theta\sim n P_\Theta+2 E\sim n H_\Theta-(n-2)E
\end{equation}
the linear system $|( nH-2\Theta)|_\Theta|$ is sent by $\phi_{| P_\Theta+E|}$ to the linear system $\mathfrak L$  in ${\mathbb P^3}^*$ of surfaces of degree $n$ containing $R^*$ with multiplicity $n-2$. So $S_\Theta$ can be regarded as the normalization of a surface of degree $n$ in ${\mathbb P^3}^*$ with points of multiplicity $n-2$ along $R^*$, with the constraint that $S_\Theta$ cuts on $E$ the same curve $C$ that $S_A$ cuts on $E$, which, as we saw in Lemma \ref {lem:sit}, is a smooth curve in the linear system $|2{\bf r}_E+(n-2){\bf f}_E|$. 

In what follows we will be interested in deformations of a surface of the type of $S_0$ in $\mathcal X$. If $S_\Theta$ is smooth, such a deformation can be interpreted as a total smoothing of $S$ in $\mathbb P^3$ (see \cite [Prop. 3.2] {CG}).

\subsection{The family of surfaces with a double line} \label{ssec:fam}

In this section we prove the following:

\begin{proposition}\label{prop:fam} The family $\mathcal F_n$ of surfaces of degree $n$ in $\mathbb P^3$ that are singular along a line is irreducible, of dimension 

$$
f_n={{n+2} \choose 3}+\frac {n(n-3)}2+3.
$$

The general surface in $\mathcal F_n$ has ordinary singularities along a line and no other singularity. 
\end{proposition}

\begin{proof} Let $r\subset \mathbb P^3$ be a line, and let $\mathcal S_n$ be the linear system of surfaces of degree $n$ in $\mathbb P^3$ that are singular along $r$. It is easy to see, and can be left to the reader, that this linear system has dimension

$$
s_n={{n+2} \choose 3}+\frac {n(n-3)}2-1.
$$

Then the first assertion easily follows and the final assertion is clear. \end{proof}

\subsection{Nodal deformations of surfaces with a double line}  In view of Problem \ref {prob:main}, we are not interested in a total smoothing of $S$, but in deformations that are nodal, with possibly the maximum possible number of nodes. To obtain such a deformation, we will need to impose to $S_\Theta$ to have nodes off $E$, and possibly the maximum possible number of nodes.  In order to be sure that we can preserve these nodes in a deformation off the central fiber, we should be able to apply \cite [Thm. 4.6]{CG}, so
it will be necessary to know that the nodes of $S_\Theta$ impose independent conditions to the linear system of surfaces $|( nH-2\Theta)|_{\mathcal X_0}|$.

Recall that $\Theta$ can be seen as the blow--up of $\phi:=\phi_{|P_\Theta+E|}: \Theta \longrightarrow {\mathbb P^3}^*$ along $r:=R^*$ via the morphism determined by the linear system $|P_\Theta+E|$. As we saw in Section \ref {ssec:kip}, the image $\Sigma$ of $S_\Theta$ via the map $\phi$ is a surface of degree $n$ in $ {\mathbb P^3}^*$ that contains the line $r$ with multiplicity $n-2$. 
Suppose that $S_\Theta$ contains $\delta$ nodes. Then $\Sigma$ contains $\delta$ nodes off the line $r$.

\begin{lemma}\label{lem:lkp} Suppose that $S_\Theta$ contains $\delta$ nodes. Then they impose independent conditions to the linear system of surfaces $|( nH-2\Theta)|_{\mathcal X_0}|$ if and only if the $\delta$ nodes of $\Sigma$ impose independent conditions to the linear system of surfaces of degree $n$ in 
$ {\mathbb P^3}^*$ passing through the line $r$ with multiplicity $n-2$. 
\end{lemma} 

\begin{proof} This is an immediate consequence of \eqref {eq:hki}. \end{proof}

\begin{theorem}\label{thm:m3} Let $S\subset \mathbb P^3$ be a general surface of degree $n\geq 3$ that has only {ordinary singularities} along a line $R$. Then $S$ belongs to the closure of a regular component of $V^{\mathbb P^3, |\mathcal O_{\mathbb P^3}(n)|}_\delta$, with $\delta=3n-4-\epsilon$, where $0\leq \epsilon\leq 1$ has the same parity of $n$. In particular, $\mathcal F_n$ is contained in $V^{\mathbb P^3, |\mathcal O_{\mathbb P^3}(n)|}_\delta$. 
\end{theorem}

\begin{proof} Let us consider the situation described in Sections \ref {ssec:lio} and  \ref {ssec:kip}. In $\mathcal X_0$ we have the surface $S_A$ that we can identify as the minimal resolution $f: S'\longrightarrow S$. Recall that $S_A$ cuts on $E=A\cap \Theta$ the smooth curve $C$. By Proposition \ref {prop:kop}, by the construction in Section \ref {ssec:constr}, by Propositions \ref {prop:lopg1} and \ref {prop:ojo}, there is a surface $\Sigma$ of degree $n$ in $\mathbb P^3$ with mulptiplicity $n-2$ along a line $r$ and with $2\ell$ nodes (with $3n-4=2\ell+\epsilon$ and  $0\leq \epsilon\leq 1$), such that, if $S''\longrightarrow \Sigma$ is the partial minimal desingularization of $\Sigma$ along $r$, then the pull--back of $r$ on $S''$ is the curve $C$. Then we can glue $S_A$ with $S_\Theta=S''$ along $C$, thus obtaining a Cartier divisor $S_0=S_A\cup S_\Theta\subset \mathcal X_0$. Now, taking into account Lemma \ref {lem:lkp} and the fact that the $2\ell$ nodes of $\Sigma$ impose independent conditions to the linear system of surfaces of degree $n$ that contain $r$ with multiplicity $n-2$ (see Proposition \ref {prop:ojo}), we can apply \cite [Thm. 4.6]{CG} and conclude the proof of the theorem.\end{proof} 

The following remarks are in order.

\begin{remark}\label{rem:limit} The proof of Theorem \ref {thm:m3} shows that $S$, as in the statement of the theorem,  can be deformed to a $\delta$--nodal surface (with $\delta=3n-4-\epsilon=2\ell$, where $0\leq \epsilon\leq 1$) in such a way that $\ell$ (arbitrary chosen) pinch points of $S$ deform each to a pair of nodes.   This follows by the particular configurations of the nodes of the surface $S_\Theta$ (cf. Proposition \ref{prop:ojo} and Lemma \ref{lem:axx}).

\end{remark}

\begin{remark}\label{rem:dec} The codimension of the family $\mathcal F_n$ in $|\mathcal O_{\mathbb P^3}(n)|$ is
$$
{{n+3}\choose 3}-1-f_n=3n-3.
$$
We proved that there is a regular component $V$ of $V^{\mathbb P^3, |\mathcal O_{\mathbb P^3}(n)|}_\delta$ contaning $\mathcal F_n$, and therefore $V$ has codimension  $\delta=3n-4-\epsilon$. So, if $n$ is even, so that $\epsilon=0$, we see that $V$ has dimension just one more than $\mathcal F_n$. It follows that $3n-4$ is the maximum $\delta$ such that $\mathcal F_n$ is contained in a \emph{regular component} of $V^{\mathbb P^3, |\mathcal O_{\mathbb P^3}(n)|}_{\delta}$. In the case $n$ odd we have  $\delta=3n-5$, thus  $V$ has dimension  2 more than $\mathcal F_n$ so that there is the possibility that $\mathcal F_n$ is contained in a regular component of  $V^{\mathbb P^3, |\mathcal O_{\mathbb P^3}(n)|}_{3n-4}$. This is an interesting question to be investigated. 
\end{remark}

{}
\end{document}